\documentclass[11pt]{amsart}

\usepackage{enumitem}
\usepackage{verbatim}
\usepackage{hyperref}
\usepackage{cleveref}
\usepackage{graphicx}
\usepackage{xcolor}

\usepackage{amssymb}
\usepackage{amsmath}
\usepackage{amsthm}

\newcommand{\R}{\mathbb{R}}

\newcommand{\grad}{\mathrm{grad}}

\newcommand{\Span}{\mathrm{span}}

\newtheorem{theorem}{Theorem}[section]

\newtheorem{proposition}[theorem]{Proposition}

\newtheorem{definition}[theorem]{Definition}

\newtheorem{thma}{Theorem}
\newtheorem*{thma*}{Theorem}

\theoremstyle{remark}
\newtheorem{remark}[theorem]{Remark}

\newtheorem*{conjecture}{Conjecture}

\newcommand{\ppn}[1]{\left.\frac{\partial}{\partial #1} \right|_{#1=0}}
\newcommand{\ppnn}[2]{\left.\frac{\partial^2}{\partial #1  \partial #2} \right|_{#1=#2 =0}}

\DeclareMathOperator{\Int}{int}
\numberwithin{equation}{section}
\numberwithin{figure}{section}

\begin{document}
	
\UseRawInputEncoding

\title[Geodesic Anosov flows]{Geodesic Anosov flows, hyperbolic closed geodesics and stable ergodicity}
\author{Gerhard Knieper}
\address{Dept.\ of Mathematics, Ruhr University Bochum, 44780 Bochum, Germany}
\email{gerhard.knieper@rub.de}

\author{Benjamin H. Schulz}
\address{Dept.\ of Mathematics, Ruhr University Bochum, 44780 Bochum, Germany}
\email{benjamin.schulz-c95@rub.de}

\date{\today}

\begin{abstract}

In this paper we show that the geodesic flow of a Finsler metric is Anosov if and only if there exists a $C^2$ open neighborhood of Finsler metrics all of whose closed geodesics are hyperbolic.
For surfaces this result holds also for Riemannian metrics.  This follows from a recent result
of Contreras and Mazzucchelli.
Furthermore, geodesic flows of Riemannian or Finsler metrics on surfaces are $C^2$ stably ergodic if and only if they are Anosov. 
\end{abstract}

\thanks{The authors are partially supported by the German Research Foundation (DFG),
CRC TRR 191, \textit{Symplectic structures in geometry, algebra and dynamics.}}
\maketitle

\section{Introduction}
Let $V$ be a closed manifold  and  $\| \cdot \|$ be a norm induced by a smooth Riemannian metric on $V$.
A  smooth fixed point free flow $\phi^t: V \to V$ on $V$ is called Anosov if there exist constants $k, C > 0$ and  a continuous flow invariant splitting 
$$
 TSM =E^s \oplus E^u\oplus E^0,
$$
where $ E^0$ is the 1-dimensional bundle given by the span of the infinitesimal generator of $\phi^t$
and where
$$
\|D \phi^t(p) v\| \le C \cdot e^{-kt} \| v \|
$$
for all $v \in E^s(p), t \ge 0$, as well as
$$
\| D\phi^{-t} (p) v \| \le C e^{-kt} \|v \|
$$
for all $v \in E^u(v)$, $t \ge 0$. It is well known that Anosov flows are structurally stable, ergodic and even mixing with respect to a smooth flow invariant measure \cite{A67}. Furthermore, the growth rate of periodic orbits is exponential. A particularly interesting and important class are given by Anosov flows induced by a Riemannian or Finsler metric.
In particular, if $(M,g)$ is a closed Riemannian manifold of negative curvature the geodesic flow on the unit tangent bundle is Anosov (see e.g.  \cite{gK02} for a proof).
 A remarkable theorem due to Klingenberg \cite{wk71}  asserts that such metrics do not have conjugate points.
For Finsler metrics this is due to \cite{PP94}. A main result of our paper is the following characterization of Anosov flows induced by Finsler metrics on closed manifolds and Riemannian metrics on closed surfaces in terms of periodic orbits.
\begin{thma}\label{thm:charact.}
Let $M$ be a closed manifold and denote by $\mathcal F_{\mathrm{hyp}}(M)$  the set of Finsler metrics all of whose closed geodesic are hyperbolic. Then the $C^2$ interior of $\mathcal F_{\mathrm{hyp}}(M)$ 
are exactly the set of Finsler metrics whose geodesic flow is Anosov. 

Moreover, if $M$ is a closed surface and $\mathcal R_{\mathrm{hyp}}(M)$ is the set of Riemannian metrics all of whose closed geodesic are hyperbolic then the $C^2$ interior of
$\mathcal R_{\mathrm{hyp}}(M)$ are exactly the set of Riemannian metrics whose geodesic flow is Anosov. 
\end{thma}
\begin{conjecture}
We conjecture that Theorem \ref{thm:charact.} should also hold for reversible Finsler metrics or Riemannian metrics on closed manifolds with arbitrary dimension.
\end{conjecture}
As an application we show for surfaces that stable ergodicity (or stable transitivity) of the geodesic flow is equivalent to the Anosov property.
\begin{thma} \label{thm:stable-ergodicity}
Let $M$ be a closed surface and and denote by $\mathcal F_{\mathrm{erg}}(M)$  resp. $\mathcal R_{\mathrm{erg}}(M)$ the set of Finsler  resp.  Riemannian metrics whose geodesic flows are ergodic. 
 Then the $C^2$ interior of $\mathcal F_{\mathrm{erg}}(M)$  resp. $\mathcal R_{\mathrm{erg}}(M)$ are exactly the set of Finsler resp. Riemannian metrics whose geodesic flows are Anosov. The same holds if we replace ergodic metrics by metrics whose geodesic flows are topologically transitive.
\end{thma} 

\section{Generic properties of Hamiltonian systems}
Let $(V,\omega)$ be a symplectic manifold, i.e. $V$ is a smooth manifold and $\omega$ is a non-degenerate smooth skew symmetric and closed two form. Since  $\omega$ is non-degenerate, $V$ is of even dimension $2n$.
 Let $H: V \to \R$ be a $C^{r+1}(V)$ time independent Hamiltonian function. This induces a $C^r$ Hamiltonian vector field $X_H$ defined by
$$
\omega_x(X_H(x), v) =dH(x)(v)
$$
for all $x \in V$ and $v \in T_xV$. We denote by $\phi_{H}^t$ the Hamiltonian flow induced by the vector field $X_H$ and for $c \in H(V)$ we denote
by $V_H(c) = H^{-1}(c)$ the associated energy level. Since   
$$ 
0 = \omega_x(X_H(x), X_H(x)) = dH(x)(X_H(x)) = \ppn t \phi_{H}^t(x),
$$
  we have that each energy surface is invariant under the Hamiltonian flow $\phi_{H}^t$.
If $c$ is a regular value of $H$ then  $V_H(c)$ is a  $C^{r-1}$ manifold of odd dimension. Furthermore $R = \{R_x\}_{x \in V_H(c)}$ given by 
$$
R_x :=\{ v \in T_x V_H(c) \mid \omega_x(v, w) = 0, \; \text{for all} \; w \in T_xV_H(c) \}
$$
defines a line bundle, called the characteristic line bundle of $(\omega, V_H(c))$. In particular, $X_H(x) \in R_x $ and if $c$ is regular $X_H(x) \not= 0$ for all $x \in  V_H(c)$ and therefore generates the characteristic line field $R$.
Furthermore, if $L_x \subset  T_x V_H(c)$ is a complement of $R_x$ then $ \omega_x: L_x \times L_x \to \R$ is non-degenerate.\\
It follows from the above discussion that for two Hamiltionians $H, H' \in C^2(V)$ on a symplectic manifold $(V, \omega)$ with regular values $c, c'$ and with the same energy surface $V_H(c) = V_{H'}(c')$ the Hamiltonian vector fields 
 $X_H$ and  $X_{H'}$  are linearly dependent on $V_H(c)$ and that therefore the orbits of the associated Hamiltonian flows agree up to parametrization.\\
 
Under certain conditions $V_H(c)$ is a contact manifold.
\begin{definition}
A contact manifold is a pair $(M, \theta)$ consisting of a $2n-1$ dimensional manifold $M$ together with a one form $\theta$ such that $\theta \wedge (d\theta)^{n-1}$ defines a volume form on $M$.
Furthermore, the vector field $R_\theta$ is called the Reeb vector field if $\theta(R_\theta) =1$ and $d\theta(R_\theta, \cdot) = 0$.
\end{definition}
\begin{proposition}
Let $(V,\omega)$ be an exact symplectic manifold, i.e. there is a one form $\theta$ such that $\omega =d \theta$.  Let $H \in C^2(V)$ be a Hamiltonian function and $c \in \R$ be a regular value of $H$ such that 
$\theta(X_H(x)) \not= 0$ for all $x \in V_H(c)$. Then $(V_H(c), \theta_{\mid V_H(c) })$ is a contact manifold.
\end{proposition}
\begin{proof}
Since $\theta(X_H(x)) \not= 0$, we have for $L_x =\ker \theta_{x_{\mid V_H(c) }} $
$$
L_x\oplus \Span \{X_H(x) \} = T_xV_H(c)
$$
and, since $\omega_x =d \theta_x$, the two form $d \theta $ is non-degenerate on $L_x$. Since $\theta(X_H(x)) \not= 0$, we obtain that $\theta \wedge d \theta^{n-1} $ is non zero on $T_x  V_H(c)$ for all $x \in V_H(c)$.
\end{proof}

A point $x \in V$ is called periodic if there exists $T>0$ such that $\phi_H^T(x) = x$. The period $T$ is called prime period if $\phi_H^t(x) \not= x$ for all $t \in (0,T)$.
A periodic point $x \in V_H(c)$  with prime period $T$ is called quasi-elliptic if  the linearized Poincar\'e map $D\phi_H^T(x): L_x \to  L_x$ restricted to the contact structure contains a non real eigenvalue with norm one.
If all eigenvalues of  $D\phi_H^T(x): L_x \to  L_x$   do not have norm one the periodic point is called hyperbolic. This notions do not depend on the choice of a chosen point on the orbit
 $\phi_H^t(x)$.\\
 We start to recall results of Newhouse \cite{SN77} on the generic existence of quasi-elliptic and hyperbolic closed orbits for Hamiltonian systems. 
 
\begin{theorem}\label{thm:Newhouse1}
Let $(V,\omega)$ be a symplectic manifold with $\dim M =2n \ge 4$ and $H_0: V \to \R$ be a smooth Hamiltonian with a compact regular energy level $H^{-1}_0(1)$. Then there exists a $C^2$ neighborhood $U \subset C^{\infty}(V)$ of $H_0$ and a dense set
$Q \subset U$ such that for all $H \in Q$ either the Hamiltonian flow $\phi^t_H$ on $H^{-1}(1)$ is Anosov or has a closed quasi-elliptic orbit.
\begin{remark}
 Newhouse  proves this theorem in details in the case of Hamiltonian diffeomorphisms in Theorem 1.3 of his paper. For the flow case he uses the work \cite{Ta70} of Takens. 
\end{remark}

\end{theorem}
\section{Co-geodesic Finsler flows and Hamiltonian systems}
Let $T^*M$ be the cotangent bundle of a smooth manifold $M$ and $\pi: T^*M \to M$ the canonical projection. Then the canonical $1$-form $\alpha$ on  $T^*M$ is defined by
$$
\alpha_v(\xi) = v( D\pi(v)( \xi))
$$
for $v \in T^*M$ and $\xi \in  T_vT^*M$. Furthermore, $ \omega = -d \alpha$ defines the canonical symplectic structure on $T^*M$. 
In order to relate Finsler metrics to Hamiltonians on $T^*M$ it is convenient to define them on the cotangent bundle. A Finsler co-metric is a continuous map $F: T^*M \to \R$ which is smooth outside the zero section. Furthermore,
\begin{enumerate}
\item
$F(v) \geq 0 $ and $F(v) = 0$ if and only if $v = 0$,
\item $F$ is 1-homogeneous i.e.
$F( \lambda v) = \lambda F(v)$ for all $\lambda \ge 0$,
\item for each $u \in T^*M$ the symmetric $2$-form $g_u^F: T_{\pi u}^*M \times T_{\pi u}^*M \to \R$ given by
$$
g_u^F(v,w) :=\ppnn t s F^2(u +sv +t w)
$$
is positive definite.
\end{enumerate}
A  Finsler co-metric $F: T^*M \to \R$ is called reversible if 
$$
F(v) = F(-v)
$$
for all $v \in T^*M$.
We denote by
$$
S^*M :=\{ v \in T^*M \mid F(v) = 1 \}
$$
the unit cotangent bundle of $F$. Furthermore, the canonical $1$-form $\alpha$ defines a contact structure on $S^*M$. The Reeb flow generated by the Reeb vector field of the contact structure is called the co-geodesic flow. 
 Furthermore, the Reeb vector field coincides with the Hamiltonian vector field $X_H$ of the Hamiltonian $H = \frac{1}{2}F^2$ on the symplectic manifold $(T^*M, -d \alpha)$ restricted to the energy shell $H^{-1}(\frac{1}{2})$.
Using the Legendre transform one can identify Finsler co-metrics with Finsler metrics on $TM$. For more details on Finsler metrics see \cite{Bao2000}, \cite{Chern2005}. \\

Now we provide a condition which implies that the Hamiltonian flow is up to a time change the geodesic flow of some Finsler metric $F:  T^*M \to \R$.
\begin{proposition}\label{prop:metrics}
Let $M$ be a compact manifold, $H: T^*M \to \R$ be a fiberwise strictly convex Hamiltonian such that the zero section is contained in the sublevel set $H^{-1}((-\infty,1))$ and that $\pi H^{-1}(1) =M$ holds.
Then there exists a Finsler metric $F:  T^*M \to \R$ such that
${(F^2)}^{-1}(1) = H^{-1}(1)$. In particular, the Finsler flow of $F$ agrees with the Hamiltonian flow $H$ on $H^{-1}(1)$ up to reparametrization.  Moreover, if $H(-u) =1$ for all $u \in  H^{-1}(1)$ the Finsler metric is reversible.
\end{proposition}
\begin{proof}
Define for each $\lambda \geq 0$ and $u \in H^{-1}(1)$
$$
F( \lambda u) = \lambda.
$$
The function $h_u : (0,\infty) \to \R$, $t \mapsto H(tu)$ is strictly convex i.e. $h_u'' > 0$ and thereby $h'_u$ is strictly increasing.
By assumption $h_u(0) < h_u(1) = 1$. Hence, due to the mean value theorem, $h_u(t) < 1$ for all $t \in (0,1)$ and $h_u(t) > 1$ for all $t > 1$.
This implies that $F$ is well defined.\\
Since  by the definition of $F$ the level sets $H^{-1}(1)$ and $(F^2)^{-1}(1)$ coincide, the fiberwise gradient of $H$ and $F^2$ with respect to some Riemannian co-metric on $M$ are linearly dependent at points in $H^{-1}(1)$, i.e. there exists a function $f : H^{-1}(1) \to \R$ such that
$\grad F^2(u) = f(u) \grad H(u)$.
The fiberwise strict convexity of $F^2$ can thereby be traced back to that of $H$. In particular, for all $u \in H^{-1}(1)$ and all $v,w \in T_{\pi u}^*M$ it is
\begin{align*}
\langle D_w \grad F^2(u) , v \rangle
= f(u) \, \langle D_w \grad H(u) , v \rangle.
\end{align*}
The fiberwise strict convexity of $H$ and the $2$-homogeneity of $F^2$ imply
\begin{align*}
\langle \grad H(u) , u \rangle = h_u'(1) > 1, && f(u) \langle \grad H(u) , u \rangle = \langle \grad F^2(u) , u \rangle = 2.
\end{align*}
Hence, $f(u)$ is positive and therefore $F^2$ is fiberwise strictly convex.
In particular, $F$ is a Finsler metric and by the discussion above its geodesic flow agrees up to parametrization with the Hamiltonian flow on $H^{-1}(1)$. \\
Furthermore, if $H(-u) =1$ for all $u \in  H^{-1}(1)$ then   $F(v) = F(-v)$ for all $v \in T^*M$ is the consequence of the definition of $F$.
\end{proof}
\begin{remark} We are grateful to Marco Mazzuchelli to inform us that Proposition \ref{prop:metrics} can be also deduced from Corollary 2 in \cite{CIP98}. The reason is that by assumption the zero section is contained in $H^{-1}((-\infty,1))$ which implies that $1$ is larger than Ma\~{n}\'{e}'s
critical value. However, since our approach is more elementary and explicit we keep the proof for the convenience of the reader.
\end{remark}
\section{Closed geodesics and characterization of Anosov metrics}
The purpose of this section is to prove  Theorem  \ref{thm:charact.} and Theorem \ref{thm:stable-ergodicity} stated in the introduction.
In the proofs we will use the following characterization of Anosov metrics.
\begin{theorem} \label{thm:conj1}
Let $M$ be a closed manifold and denote by $\mathcal R (M) $ resp. $\mathcal F(M)$ the set of  all Riemannian resp. Finsler metrics without conjugate points. Then the $C^2$ interior of $\mathcal R (M)$ resp.  $\mathcal F (M)$
are precisely the Anosov metrics.
\end{theorem}
\begin{proof}
As mentioned above Anosov Riemannian or Finsler metrics have no conjugate points. From structural stability follows that these metrics are $C^2$ open. The reverse assertion was proved by Ruggiero \cite{Ru91} using a characterization of Anosov flows by Eberlein \cite{Eb73}. 
His result has been extended by Contreras et al. in \cite{CIS98} for Finsler metrics.
\end{proof}
\begin{proof}[Proof of Theorem \ref{thm:charact.} in the Finsler case.]
Let $M$ be a closed manifold and
$\Int(\mathcal F_{\mathrm{hyp}}(M))$ be the $C^2$ interior of the set of all Finsler co-metrics all of whose closed geodesics are hyperbolic. Let $V = T^*M$ be the cotangent bundle equipped
with the canonical symplectic structure $\omega$.
Consider a Finsler co-metric $F \in \Int(\mathcal F_{\mathrm{hyp}}(M))$. By Theorem \ref{thm:Newhouse1} of Newhouse
there is a sequence of Hamiltonian functions $H_n$ which converges in the $C^2$ topology to $F$ such that the Hamiltonian flow of $H_n$ on the energy surface $H_n^{-1}(1)$ is either Anosov or contains a quasi-elliptic closed orbit.
We can assume that $H_n$ is fiberwise strictly convex, that $0 \in H^{-1}((0,1))$ and that $\pi H^{-1}(1) = M$. Let $F_n$ be the Finsler metrics provided by Proposition \ref{prop:metrics}. Since they also converge in the $C^2$ topology to $F \in \Int(\mathcal F_{\mathrm{hyp}}(M))$ we can assume that all closed orbits of the co-metrics $F_n$ are hyperbolic. But, since the Hamiltonian flow of $H_n$ on  $H_n^{-1}(1)$ agrees up to a time change with
the geodesic Finsler flow of $F_n$, all closed orbits on $H_n^{-1}(1)$ are hyperbolic as well and hence the flow is Anosov. In particular, the flows $\phi_{F_n}^t$ are Anosov and therefore free of conjugate points.
Since the set of metrics without conjugate points is $C^2$ closed, we obtain that  the metric $F$ has no conjugate points as well.
Due to  Theorem  $\ref{thm:conj1}$, the set  $\Int(\mathcal F_{\mathrm{hyp}}(M))$  consists of Anosov metrics. 
\end{proof}
\begin{remark}
Note that it is not possible to deduce Theorem \ref{thm:charact.}  from Newhouse's Theorem \ref{thm:Newhouse1}
in the Riemannian case. 
However, for surfaces there is a replacement by the following theorem of Contreras and Mazzucchelli  which is proved among other things in \cite{CM}.
\end{remark}
\begin{theorem} \label{thm:CM}
Let $M$ be a closed surface. Then there exists a $C^2$ dense set of Riemannian metrics $Q$ such that for all $g \in Q$ either its geodesic flow is Anosov or
contains a non-hyperbolic closed orbit.
\end{theorem}
\begin{remark}
For Riemannian metrics which do not contain contractible closed geodesics this result has been also obtained by Schulz \cite{Sch21}.
\end{remark}

Let $(M,F)$ be a Finsler manifold and let $v$ be
an elliptic periodic point, i.e. $v$ is an elliptic fixed point 
of every Poincar\'{e} map. Let $p$ be the local representation of a Poincar\'{e} map $P$ with respect to a symplectic local chart such that zero is an elliptic fixed point of $p$. There exists a symplectomorphism $\sigma$ with respect to the standard symplectic form on $\R^2$ such that $q := \sigma \circ p \circ \sigma^{-1}$ is of the form
\begin{align*} \textstyle
q (x) = 	A_{\varphi(x)} (x) + \mathcal{O}(\Vert x \Vert^4),
\end{align*}
where $A_{\varphi}$ denotes the rotation around zero in $\R^2$ with the angle $\varphi \in \R$ and where it is $\varphi(x) = \alpha + \beta \cdot \Vert x \Vert^2$ for constants $\alpha, \beta \in \R$. The local representation $q$ is called the Birkhoff normal form of $P$ and is independent of both the transverse cross section for $P$ and the initial choice of the symplectic local chart. The constants $\alpha$ and $\beta$ are symplectically invariant and thus are called Birkhoff invariants. A proof can be found in \cite{Mo77}. If $\beta$ is non-zero, then $v$ is called of twist type.\\
For an elliptic periodic point $v$ of twist type, due to Siegel and Moser \cite{SM71}, for every $\delta > 0$ there exists a simple closed curve $\gamma$ in the open $\delta$-ball at zero in $\R^2$ which has winding number one around zero and which is invariant under $q$. This gives rise to a flow-invariant torus
which divides the unit tangent bundle into two flow-invariant open sets of positive Liouville measure. Hence, an elliptic periodic point of twist type prevents the geodesic flow from being ergodic and, in particular, from being topologically transitive.\\

Using this we are able to characterize stably ergodic Riemannian or Finsler geodesic flows as stated in Theorem \ref{thm:stable-ergodicity} in the introduction.

\begin{proof}[Proof of Theorem \ref{thm:stable-ergodicity}]
Suppose $F \in \Int(\mathcal F_{\mathrm{erg}}(M))$ has a closed non hyperbolic geodesic.
Due to Rademacher and Taimanov \cite{RT20}, there exists $F' \in U$ such that $F'$ has an elliptic orbit of twist type. Their result is based on a perturbation result of Carballo and Miranda \cite{CM13} for Tonelli Hamiltonian flows. But then, due to the discussion above, the geodesic flow of $F'$ is neither ergodic nor topologically transitive. Hence all closed orbits of the geodesic flow of Finsler metrics in $\Int(\mathcal F_{\mathrm{erg}}(M))$ are hyperbolic and by Theorem \ref{thm:charact.} the metrics in $\Int(\mathcal F_{\mathrm{erg}}(M))$ are all Anosov. The proof in the Riemannian case proceeds similarly and uses the perturbation result by Klingenberg and Takens \cite{KT72} (see also \cite{Kli78}).
\end{proof}

\bibliographystyle{amsalpha}

\bibliography{Anosov}

\end{document}